\documentclass[a4j,12pt,draft]{article}
\usepackage{amsmath}
\usepackage{amssymb}
\usepackage{amsthm}
\usepackage{amscd}
\newtheorem{theorem}{Theorem}[section]
\newtheorem{lemma}{Lemma}[section]
\theoremstyle{definition}
\newtheorem{definition}{Definition}[section]

\theoremstyle{remark}
\newtheorem{remark}{Remark}[section]
\theoremstyle{corollary}
\newtheorem{corollary}{Corollary}[section]

\begin{document}

\title{The equivariant  simplicial de Rham complex and the classifying space of a semi-direct product group}
\author{Naoya Suzuki}
\date{}
\maketitle

\begin{abstract}
We show that the cohomology group of the total complex of the equivariant simplicial de Rham complex is isomorphic to the cohomology group of the classifying space of a semi-direct product group. 
\end{abstract}

\section{Introduction}
In \cite{Wein}, Weinstein introduced the equivariant version of the simplicial de Rham complex. That is a double complex 
whose components are equivariant differential forms which is called the Cartan model(\cite{Ber}). Weinstein expected that the cohomology group of its
total complex is isomorphic to $H^*(B(G \rtimes H))$. Here $B(G \rtimes H)$ is the classifying space of a semi-direct product group.
In this paper, we show this conjecture is true.
\\

\section{Review of the simplicial de Rham complex}
In this section we recall the relation between the simplicial manifold $NG$ and the classifying space $BG$. We also recall
the notion of the equivariant version of the simplicial de Rham complex.

\subsection{The double complex on simplicial manifold}

For any Lie group $G$, we have simplicial manifolds $NG$, $N \bar{G}$ and simplicial $G$-bundle  $\gamma : N \bar{G} \rightarrow NG$
as follows:\\
\par
$NG(q)  = \overbrace{G \times \cdots \times G }^{q-times}  \ni (g_1 , \cdots , g_q ) :$  \\
face operators \enspace ${\varepsilon}_{i} : NG(q) \rightarrow NG(q-1)  $
$$
{\varepsilon}_{i}(g_1 , \cdots , g_q )=\begin{cases}
(g_2 , \cdots , g_q )  &  i=0 \\
(g_1 , \cdots ,g_i g_{i+1} , \cdots , g_q )  &  i=1 , \cdots , q-1 \\
(g_1 , \cdots , g_{q-1} )  &  i=q
\end{cases}
$$

\par
\medskip
$N \bar{G} (q) = \overbrace{ G \times \cdots \times G }^{q+1 - times} \ni (\bar{g}_1 , \cdots , \bar{g}_{q+1} ) :$ \\
face operators \enspace $ \bar{\varepsilon}_{i} : N \bar{G}(q) \rightarrow N \bar{G}(q-1)  $ 
$$ \bar{{\varepsilon}} _{i} (\bar{g}_1 , \cdots , \bar{g}_{q+1} ) = (\bar{g}_1 , \cdots , \bar{g}_{i} , \bar{g}_{i+2}, \cdots , \bar{g}_{q+1})  \qquad i=0,1, \cdots ,q $$

\par
\medskip

We define $\gamma : N \bar{G} \rightarrow NG $ as $ \gamma (\bar{g}_1 , \cdots , \bar{g}_{q+1} ) = (\bar{g}_1 {\bar{g}_2}^{-1} , \cdots , \bar{g}_{q} {\bar{g}_{q+1}}^{-1} )$.\\

(The standard definition also involves degeneracy operators but we do not need them here).\\

\begin{remark}
Here we use the notation $\bar{g}_i$ to distinguish elements in $NG$ from elements in $N \bar{G}$. It does not mean the
complex conjugate.\\
\end{remark}

For any simplicial manifold $\{ X_* \}$, we can associate a topological space $\parallel X_* \parallel $ 
called the fat realization defined as follows. 
$$  \parallel X_* \parallel \enspace \buildrel \mathrm{def} \over = \coprod _{n}  {\Delta}^{n} \times X_n / \enspace ( {\varepsilon}^{i} t , x) \sim (  t , {\varepsilon}_{i} x).$$
Here ${\Delta}^{n}$ is the standard $n$-simplex and ${\varepsilon}^{i}$ is a face map of it.
It is well-known that 
$\parallel \gamma \parallel : \parallel N \bar{G} \parallel \rightarrow \parallel NG \parallel$ is the universal bundle $EG \rightarrow BG$  (see \cite{Dup2} 
\cite{Mos} \cite{Seg}, for instance). \\

Now we introduce a double complex associated to a simplicial manifold.

\begin{definition}
For any simplicial manifold $ \{ X_* \}$ with face operators $\{ {\varepsilon}_* \}$, we have a double complex ${\Omega}^{p,q} (X) \buildrel \mathrm{def} \over = {\Omega}^{q} (X_p) $ with derivatives as follows:
$$ d' = \sum _{i=0} ^{p+1} (-1)^{i} {\varepsilon}_{i} ^{*}  , \qquad  d'' := (-1)^{p} \times {\rm the \enspace exterior \enspace differential \enspace on \enspace }{ \Omega ^*(X_p) } .$$

\end{definition}
$\hspace{30em} \Box $
\newpage
For $NG$ and $N \bar{G} $ the following holds.

\begin{theorem}[\cite{Bot2} \cite{Dup2} \cite{Mos}]
 There exist ring isomorphisms 

$$ H^*({\Omega}^{*} (NG))  \cong  H^{*} (BG ), \qquad  H^*({\Omega}^{*} (N \bar{G})) \cong H^{*} (EG ).  $$

 Here ${\Omega}^{*} (NG)$  and  ${\Omega}^{*} (N \bar{G})$  mean the total complexes.
\end{theorem} 
$\hspace{30em} \Box $

\subsection{Equivariant version}

When a Lie group $H$ acts on a manifold $M$, there is the complex of equivariant differential forms 
${\Omega}_H ^{*} (M) := ( {\Omega} ^{*} (M) \otimes S(\mathcal{H}^*))^H$ with suitable differential $d_H$ (\cite{Ber} \cite{Car}).
Here $\mathcal{H}$ is the
Lie algebra of $H$ and
$S(\mathcal{H}^*)$
is the algebra of polynomial functions on $\mathcal{H}$. This is
called the Cartan Model.
When $M$ is a Lie group $G$, we can define the double complex ${\Omega}^{*} _H (NG(*))$ in the same way as in Definition 2.1. 
This double complex is originally introduced by Weinstein in \cite{Wein}.

\section{The triple complex on bisimplicial manifold}
In this section we construct a triple complex on a bisimplicial manifold.\\
\par
A bisimplicial manifold is a sequence of manifolds with horizontal and vertical face and degeneracy operators which commute with each other.
A bisimplicial map is a sequence of maps commuting with horizontal and vertical face and degeneracy operators.
Let $H$ be a subgroup of $G$.
We define a bisimplicial manifold $NG(*) \rtimes NH(*)$ as follows;
\par
$$NG(p) \rtimes NH(q)  := \overbrace{G \times \cdots \times G }^{p-times} \times \overbrace{H \times \cdots \times H }^{q-times}. $$  
Horizontal face operators \enspace ${\varepsilon}_{i}^{G} : NG(p) \rtimes NH(q) \rightarrow NG(p-1)  \rtimes NH(q) $ are the same as the face operators of $NG(p)$.
Vertical face operators \enspace ${\varepsilon}_{i}^{H} : NG(p) \rtimes NH(q) \rightarrow NG(p)  \rtimes NH(q-1) $ are
$$
{\varepsilon}_{i}^{H}(\vec{g}, h_1 , \cdots , h_q )=\begin{cases}
(\vec{g}, h_2 , \cdots , h_q )  &  i=0 \\
(\vec{g}, h_1 , \cdots ,h_i h_{i+1} , \cdots , h_q )  &  i=1 , \cdots , q-1 \\
(h_{q}\vec{g}h_{q} ^{-1}, h_1 , \cdots , h_{q-1} )  &  i=q.
\end{cases}
$$
Here $\vec{g}=(g_1, \cdots , g_p)$.

We define a bisimplicial map $\gamma_{\rtimes} : N \bar{G}(p) \times N \bar{H}(q) \rightarrow NG(p) \rtimes NH(q) $ as $ \gamma_{\rtimes} (\vec{\bar{g}}, \bar{h}_1, \cdots ,\bar{h}_{q+1} ) = (\bar{h}_{q+1}\gamma (\vec{\bar{g}}) \bar{h}^{-1} _{q+1} ,
 \gamma (\bar{h}_1, \cdots, \bar{h}_{q+1}))$.
Now we fix a semi-direct product operator $\cdot_{\rtimes}$of $G \rtimes H$ as $(g, h) \cdot_{\rtimes} (g', h') := (ghg'h^{-1} , hh')$,
then
$G \rtimes H$ acts $ N \bar{G}(p) \times N \bar{H}(q)$ by right as $(\vec{\bar{g}},\vec{\bar{h}})\cdot(g,h) = (h^{-1}\vec{\bar{g}}gh, \vec{\bar{h}}h)$.
Since $\gamma_{\rtimes}(\vec{\bar{g}},\vec{\bar{h}})=
\gamma_{\rtimes}((\vec{\bar{g}},\vec{\bar{h}})\cdot(g,h))$, one can see that $\gamma_{\rtimes}$ is a principal  $(G \rtimes H)$-bundle.
$\parallel N \bar{G}(*) \times N \bar{H}(*) \parallel$ is $EG \times EH$ so its homotopy groups are trivial in any dimension 
and $\parallel NG(*) \rtimes NH(*) \parallel$ is homeomorphic to $(EG \times EH)/(G \rtimes H)$.
We can also check that $EG \times EH \rightarrow (EG \times EH)/(G \rtimes H)$ is a principal  $(G \rtimes H)$-bundle 
since $(G \rtimes H)$ is an absolute neighborhood retract (see for example \cite{Dup2} P.73).
Hence $\parallel NG(*) \rtimes NH(*) \parallel$ is a model of $B(G \rtimes H)$.

\begin{definition}
For a bisimplicial manifold $NG(*) \rtimes NH(*)$, we have a triple complex as follows:

$${\Omega}^{p,q,r} (NG(*) \rtimes NH(*)) \buildrel \mathrm{def} \over = {\Omega}^{r} (NG(p) \rtimes NH(q)) $$

Derivatives are:
$$ d' = \sum _{i=0} ^{p+1} (-1)^{i} ({{\varepsilon}^G _{i}}) ^{*}  , \qquad  d'' = \sum _{i=0} ^{q+1} (-1)^{i} ({{\varepsilon}^H _{i}}) ^{*} \times (-1)^{p} $$
$$ d''' =  (-1)^{p+q} \times {\rm the \enspace exterior \enspace differential \enspace on \enspace }{ \Omega ^*(NG(p) \rtimes NH(q)) }.$$

\end{definition}
$\hspace{30em} \Box $

Let $C^*(X)$ denote the set of singular cochains of a topological space $X$.
We can also define the triple complex ${C}^{p,q,r} (NG(*) \rtimes NH(*))$ in the same way.
Applying the de Rham theorem and the lemma below twice, we can see that the total complex ${\Omega}^{*} (NG \rtimes NH)$ of the triple complex in the Definition 3.1 is quasi-isomorphic to the total complex
 of ${C}^{p,q,r} (NG(*) \rtimes NH(*))$.
\begin{lemma}[\cite{Dup2}, lemma 1.19]
 Let $K_1 ^{p,q}$ and $K_2 ^{p,q}$ be 1.quadrant double complexes, i.e. $K_1 ^{p,q}= K_2 ^{p,q}=0$ if either $p <0$ or
$q <0$.
 Suppose $f:K_1 ^{*,*} \rightarrow K_2 ^{*,*}$  is a homomorphism of double complexes and suppose
$f^{p,q} : H^p(K_1^{*,q},d_1') \rightarrow  H^p(K_2^{*,q},d_2')$  is an isomorphism. Then also $f^* : H^*(K_1,d_1) \rightarrow  H^*(K_2,d_2)$
 is an isomorphism.$\hspace{24em} \Box $
\end{lemma}
\begin{remark}
Let $C_*(X)$ denote the set of singular chains of a topological space $X$. We can also define the triple complex ${C}_{p,q,r} (NG(*) \rtimes NH(*)):={C}_{r} (NG(p) \rtimes NH(q))$ of the singular chains in the 
same way.
\end{remark}
\section{Main theorem}
\begin{theorem}
 If $H$ is compact, there exist isomorphisms 

$$ H({\Omega}_H ^{*} (NG))  \cong H({\Omega}^{*} (NG \rtimes NH)) \cong  H^{*} (B(G \rtimes H)).$$

 Here ${\Omega}_H ^{*} (NG)$  means the total complex in subsection 2.2.
\end{theorem} 
\begin{proof}
At first we recall the Getzler's result in \cite{Get}. When a Lie group $H$ acts on a manifold $M$ by left, there is a simplicial manifold $ \{ M \rtimes NH(q) \}$ 
with face operators:
$$
{\varepsilon}_{i}(u, h_1 , \cdots , h_q )=\begin{cases}
(u, h_2 , \cdots , h_q )  &  i=0 \\
(u, h_1 , \cdots ,h_i h_{i+1} , \cdots , h_q )  &  i=1 , \cdots , q-1 \\
(h_{q} u, h_1 , \cdots , h_{q-1} )  &  i=q.
\end{cases}
$$

We need the following theorem for the proof.
\begin{theorem}[\cite{Get}]
 If $H$ is compact, there is a cochain map between the total complex of the double complex $\Omega ^*(M \rtimes NH(*))$  and $(\Omega_H ^* (M) , d_H)$ which induces an isomorphism in cohomology.
\end{theorem}
As a corollary of this theorem, we obtain the following statement.
\begin{corollary}
 For any fixed $p$,  the total complex of the double complex $\Omega ^*(NG(p) \rtimes NH(*))$  is quasi-isomorphic to $(\Omega_H ^* (G^p) , (-1)^pd_H)$
\end{corollary}
Hence using the Lemma 3.1, we can see that $H^*({\Omega}_H ^{*} (NG))$  is isomorphic to $ H^*({\Omega}^{*} (NG \rtimes NH))$.\\

Now we prove the existence of the another isomorphism.  Let $S_*(X)$ denote the set of singular simplexes of a topological space $X$.
For a triple simplicial set $S_r (NG(p) \rtimes NH(q)) $, we have the fat realization 
$$  \coprod _{r,p,q \geq 0 }   {\Delta}^{p} \times {\Delta}^{q}  \times {\Delta}^{r} \times S_r (NG(p) \rtimes NH(q) ) /  \sim .$$
with suitable identifications. This is a CW complex and the set of $n$-cells are in one-to-one correspondence with  $ \coprod _{r+p+q=n } S_r (NG(p) \rtimes NH(q) )$. 
Its homology group coincides with the homology group of the total complex of the triple complex ${C}_{p,q,r} (NG \rtimes NH)$.

So we need to show the cohomology group of this CW complex is isomorphic to $H^*(\parallel NG \rtimes NH \parallel)$.
We recall that the map $ \rho : {\Delta}^{r} \times  S_r (X) \rightarrow X$ which is defined as $ \rho(t, \sigma_r) := \sigma_r (t)$
induces an isomorphism $H_*(\coprod _{r} {\Delta}^{r} \times  S_r (X)/  \sim) \cong H_*(X)$ (see for instance \cite{Dup2} P.82). Hence for any fixed $p,q$, the following map $\rho _{p,q}$ which is same as $ \rho$ induces an isomorphism in homology.
$$ \rho _{p,q} : \coprod _{r}  {\Delta}^{r} \times  S_r (NG(p) \rtimes NH(q))  /  \sim \enspace \rightarrow NG(p) \rtimes NH(q).$$

We also use the following lemma.
\begin{lemma}[\cite{Dup2}, Lemma 5.16]
 Let $f:\{X_* \} \rightarrow \{ X'_* \} $  be a simplicial map of simplicial spaces such that  $f_p:X_p \rightarrow X_p' $  induces an isomorphism in 
homology with coefficients in a ring $\lambda $   for all $p$.
 Then $\parallel f \parallel:\parallel X_* \parallel \rightarrow \parallel X'_* \parallel$ also induces an isomorphism
in homology and cohomology with coefficients in $\lambda$.
\end{lemma}

By applying the Lemma 4.1, we see that for any fixed $p$, 
$  \parallel \rho _{p,*} \parallel : \coprod _{q}  {\Delta}^{q} \times (\coprod _{r}  {\Delta}^{r} \times  S_r (NG(p) \rtimes NH(q))  /  \sim)/  \sim 
\enspace \rightarrow \coprod _{q}  {\Delta}^{q} \times NG(p) \rtimes NH(q)  /  \sim $
induces an isomorphism in homology.\\

Hence again applying the Lemma 4.1 we can see that
$$  \parallel \rho _{*,*} \parallel :\coprod _{p}  {\Delta}^{p} \times ( \coprod _{q}  {\Delta}^{q} \times (\coprod _{r}  {\Delta}^{r} \times  S_r (NG(p) \rtimes NH(q))  /  \sim)/  \sim )/  \sim$$
$$ \enspace \rightarrow \coprod _{p}  {\Delta}^{p} \times ( \coprod _{q}  {\Delta}^{q} \times NG(p) \rtimes NH(q)  /  \sim )/  \sim .$$
 induces an isomorphism in cohomology. This completes the proof of Theorem 4.1.
\end{proof}
{\bf Acknowledgments.} \\
The author is indebted to his supervisor, Professor H. Moriyoshi for helpful discussion and good advice.
The author would like to thank the referee for his/her several suggestions to improve the present paper.

Graduate School of Mathematics, Nagoya University, Furo-cho, Chikusa-ku, Nagoya-shi, Aichi-ken, 464-8602, Japan. \\
e-mail: suzuki.naoya@c.mbox.nagoya-u.ac.jp

\begin{thebibliography}{99}
\bibitem{Ber}N. Berline, E. Getzler, and M. Vergne, Heat Kernels and Dirac Operators, Grundlehren Math. Wiss. 298, Springer-Verlag, Berlin, 1992.


\bibitem{Bot2}R. Bott, H. Shulman, J. Stasheff, On the de Rham Theory of Certain Classifying Spaces, Adv. in Math. 20 (1976), 43-56.
\bibitem{Car}H. Cartan, La transgression dans un groupe de Lie et dans un espace fibr\'{e} principal, Colloque de Topologie, CBRM Bruxelles, 1950, pp. 57-71.


\bibitem{Dup2}J. L. Dupont, Curvature and Characteristic Classes, Lecture Notes in Math. 640, Springer Verlag, 1978.

\bibitem{Get}E. Getzler, The equivariant Chern character for non-compact Lie groups, Adv. Math. 109(1994), no.1,88-107.

\bibitem{Mos}M. Mostow and J. Perchick, Notes on Gel'fand-Fuks Cohomology and Characteristic Classes (Lectures by Bott). In Eleventh Holiday Symposium. New Mexico State University, December 1973.

\bibitem{Seg}G. Segal, Classifying spaces and spectral sequences. Inst. Hautes \'{E}tudes Sci.Publ. Math. No.34 1968 105-112.

\bibitem{Wein}A. Weinstein, The symplectic structure on moduli space. The Floer memorial volume, Progr. Math., 133, Birkh\"{a}user, Basel,1995, 627-635.
\end{thebibliography}
\end{document}